\documentclass[oneside,english,reqno]{amsart}
\usepackage{algorithm}
\usepackage{algorithmic}
\usepackage{multirow}
\usepackage{subfigure}
\usepackage{enumerate}

\usepackage[T1]{fontenc}
\usepackage[latin9]{inputenc}
\usepackage{babel}
\usepackage{amsthm}
\usepackage{amssymb}
\usepackage{graphicx}
\usepackage[unicode=true,pdfusetitle,
bookmarks=true,bookmarksnumbered=false,bookmarksopen=false,
breaklinks=false,pdfborder={0 0 1},backref=false,
hidelinks]
{hyperref}

\makeatletter
\numberwithin{equation}{section}
\numberwithin{figure}{section}

\makeatother
\makeatletter
\renewcommand{\ALG@name}{Iterative Scheme}
\makeatother

\newtheorem{proposition}{Proposition}
\newtheorem{theorem}{Theorem}
\newtheorem{lemma}{Lemma}

\newtheorem{definition}{Definition}
\newtheorem{example}{Example}
\usepackage{tikz}
\usetikzlibrary{shapes.misc}
\usepackage{amssymb}
\newcommand{\gph}{\text{gph}\,}
\newcommand{\hilbertH}{{\mathcal H}}
\newcommand{\hilbertG}{{\mathcal G}}
\newcommand{\dist}{\text{dist}\,}
\newcommand{\rank}{\text{rank}\,}
\newcommand{\multifC}{{\mathbb C}}
\newcommand{\setD}{{\mathcal D}}

\begin{document}
	
	\title{On Lipschitz-like property for polyhedral moving sets}

		\author{
			Ewa M. Bednarczuk$^1$
		}
	\author{
		Krzysztof E. Rutkowski$^2$ 
	}
	\thanks{$^1$ Systems Research Institute of the Polish Academy of Sciences, Warsaw University of Technology, 		 \href{mailto:e.bednarczuk@mini.pw.edu.pl}{e.bednarczuk@mini.pw.edu.pl} .}
	\thanks{$^2$ Warsaw University of Technology, 		 \href{mailto:k.rutkowski@mini.pw.edu.pl}{k.rutkowski@mini.pw.edu.pl} .}
	\maketitle
	
	\begin{abstract}
We give sufficient conditions for Lipschitz-likeness of a class of polyhedral set-valued mappings in Hilbert spaces based on Relaxed Constant Rank Constraint Qualification (RCRCQ) proposed recently by Minchenko and Stakhovsky. To this aim we prove the $R$-regularity of
the considered  set-valued mapping and correct the respective proof given by these authors. 
	\end{abstract}

\section{Introduction}
Let $\hilbertH, \hilbertG$ be a Hilbert space and $D\subset \hilbertG$ be a nonempty set. 
Let $\multifC:\ \setD\rightrightarrows \hilbertH$ be a multifunction defined as $\multifC(p):= C(p)$, where 
\begin{align}\label{M(p,v)}%\tag{M(p,v)}
	C(p)=\left\{ x\in \hilbertH\ \bigg|\ 
	\begin{array}{ll}
		\langle x \ |\ g_i(p)\rangle = f_i(p), &i \in I_1,\\
		\langle x \ |\ g_i(p)\rangle \leq f_i(p), &i \in I_2
	\end{array}
	\right\},
\end{align}
and $f_i:\ \setD\rightarrow \mathbb{R}$, $g_i:\ \setD\rightarrow\hilbertH$, $i\in I_1\cup I_2,$ $I_1=\{1,\dots,m\}$, $I_2=\{m+1,\dots,n\}$ are Lipschitz on $\setD$ with Lipschitz constants $\ell_{f_i},\ell_{g_i}$, respectively. 

In finite dimensional case ($\hilbertH=\mathbb{R}^{n_1}$, $\hilbertG=\mathbb{R}^{n_2}$) the sufficient conditions for $R$-regularity of multifunction $\multifC$ and more general set-valued mappings have been proposed in \cite[Theorem 4]{parametric_nonlinear_programming_Minchenko}. $R$-regularity of the multifunction $\multifC$ at $ (\bar{p},\bar{x})\in \gph \multifC$ is defined as follows.
\begin{definition}%\normalfont
	Multifunction $\multifC:\ \setD\rightrightarrows \hilbertH$ given by \eqref{M(p,v)}  is said to be $R$-regular  at a point $(\bar{p},\bar{x})$, if for all $(p,x)$ in a neighbourhood of $(\bar{p},\bar{x})$,
	\begin{equation*}
		\dist\left(x,\multifC(p)\right)\leq \alpha \max\{0,\ |\langle x \ |\ g_i(p)\rangle - f_i(p)|,\ i\in I_1 ,\ \langle x \ |\ g_i(p)\rangle - f_i(p),\ i\in I_2 \}
	\end{equation*}
	for some $\alpha>0$.
\end{definition} 
The aim of the paper is to investigate the Lipschitz-like property of the multifunction $\multifC$ at $ (\bar{p},\bar{x})\in \gph \multifC$ defined as follows.
\begin{definition}%\normalfont
	Multifunction $\multifC$ is \textit{Lipschitz-like} at a point $(\bar{p},\bar{x})$, if 
	there exist a constant $\ell>0$, a neighbourhood  $U(\bar{p})$ and a neighbourhood  $V(\bar{x})$ such that for all $p_1,p_2\in U(\bar{p})$
	\begin{equation*}
		\multifC(p_1)\cap V(\bar{x})\subset \multifC(p_2)+\ell \|p_1-p_2\|\mathbb{B},
	\end{equation*}
	where $B$ denotes the open unit ball in the space $\hilbertH$.
\end{definition}

To this aim we provide Proposition \ref{theorem:RCRCQ_to_condition} which is the infinite-dimensional version of Lemma 3 of \cite{parametric_nonlinear_programming_Minchenko} applied to our set-valued mapping  \eqref{M(p,v)}.
However, the proof of  \cite[Lemma 3]{parametric_nonlinear_programming_Minchenko} which is important for the proof of  \cite[Theorem 4]{parametric_nonlinear_programming_Minchenko} is incorrect. It is also our aim to provide the correct proof of \cite[Lemma 3]{parametric_nonlinear_programming_Minchenko} in our case.

\section{Preliminaries}
Let $p\in \setD$, $w\in \hilbertH$, $w\notin C(p)$. Projection of $w$ onto $C(p)$ is defined as 
\begin{equation}\label{eq:projection}
P_{C(p)}(w)=\arg\min\limits_{x\in C(p)} \|w-x\|,
\end{equation}
or equivalently
\begin{equation}\label{eq:projection2}
P_{C(p)}(w)=\arg\min\limits_{x\in C(p)} \frac{1}{2}\|w-x\|^2.
\end{equation}

Put $f_w(x)=\|x-w\|$ and
\begin{equation*}
	f_{P_{C(p)}(w)  }^*(x)=\| x-w \|+\frac{\langle x-w \ |\ x- P_{C(p)}(w)\rangle  }{\| P_{C(p)}(w)-w\| }.
\end{equation*}
Denote $G_i(x,p)=\langle x \ |\ g_i(p)\rangle - f_i(p)$, $i\in I_1\cup I_2$ and $\bar{G}_i(x,p)=G_i(x,p)$ for $g_i(p)=a_i$, $a_i\in \hilbertH$, $i\in I_1\cup I_2$, i.e.,  $g_i$, $i\in I_1\cup I_2$ does not depend on $p$. Let $G(x,p)$ and $\bar{G}(x,p)$ be defined as
\begin{align*}
	G(x,p)=\left[ G_i(x,p)\right]_{i=1,\dots,n},\quad 
	\bar{G}(x,p)=\left[ \bar{G}_i(x,p)\right]_{i=1,\dots,n}.
\end{align*}
Let $\lambda\in \mathbb{R}^n$ and
\begin{align*}
	&L_{w}(p,x,\lambda):=f_w(x)+\langle \lambda \ |\ G(x,p) \rangle,\\
	&L_{w}^*(p,x,\lambda):=f_{P_{C(p)}(w)  }^*(x)+\langle \lambda \ |\ G(x,p) \rangle.
\end{align*}
The  sets of Lagrange multipliers corresponding to \eqref{eq:projection} are defined as
\begin{align*}
	\Lambda_{w}(p,x):=\{ \lambda\in\mathbb{R}^n\ |\ \nabla_x L_{w}(p,x,\lambda)=0,\ \lambda_i\geq 0,\ \text{and}\ \lambda_i G_i(x,p)=0,\ i \in I_2  \},\\
	\Lambda_{w}^*(p,x):=\{ \lambda\in\mathbb{R}^n\ |\ \nabla_x L_{w}^*(p,x,\lambda)=0,\ \lambda_i\geq 0,\ \text{and}\ \lambda_i G_i(x,p)=0,\ i \in  I_2  \}.
\end{align*} 
Then
\begin{align}
	\begin{aligned}\label{set:2lambda}
		&\nabla_x L_{w}(p,P_{C(p)}(w),\lambda)=\frac{P_{C(p)}(w)-w}{\|P_{C(p)}(w)-w\|}+\sum_{i=1}^{n}\lambda_i g_i(p),\\
		&\nabla_x L_{w}^*(p,P_{C(p)}(w),\lambda)=2\frac{P_{C(p)}(w)-w}{\|P_{C(p)}(w)-w\|}+\sum_{i=1}^{n}\lambda_i g_i(p).
	\end{aligned}
\end{align}
Let us note that when $w\notin C(p)$ condition $\nabla_x L_{w}(p,P_{C(p)}(w),\lambda)=0$ is equivalent to the following
\begin{align}\label{equivalence:relations_lambda}
	&\frac{w-P_{C(p)}(w)}{\|P_{C(p)}(w)-w\|}=\sum_{i=1}^{n}\lambda_i g_i(p)\quad  \Leftrightarrow \quad w-P_{C(p)}(w)=\sum_{i=1}^{n}\hat{\lambda}_i  g_i(p),
\end{align}
where $\hat{\lambda}_i=\lambda_i\|P_{C(p)}(w)-w\|$, $i=1,\dots,n$.

Let us recall that the Kuratowski limit of $\multifC$ at $\bar{p}$ is given as
\begin{equation*}
	\liminf_{p\rightarrow\bar{p}} \multifC(p)=\{ y\in \hilbertH\ |\ \forall\, p_k\rightarrow \bar{p}\ \exists\, y_k\in \multifC(p_k)\quad y_k \rightarrow y\}.
\end{equation*}
Equivalently, $\bar{x}\in \liminf\limits_{p\rightarrow\bar{p}} \multifC(p)$ if and only if 
\begin{equation}
\forall\ V(\bar{x})  \ \exists\ U(\bar{p})\ \text{s.t.}\ \multifC(p)\cap V(\bar{x})\neq \emptyset \quad \text{for}\ p\in U(\bar{p}).
\end{equation}
For any $(p,x)\in \setD\times \hilbertH$ let $I_{p}(x):= \{ i\in I_1\cup I_2\ |\ \langle x-f_i(p)\ |\ g_i(p)\rangle =0  \}$ denote the active index set for $p\in \setD$ at $x\in \hilbertH$. 
\begin{definition}[Relaxed Constant Rank Constraint Qualification] 
	The \textit{relaxed constant rank constraint qualification} (RCRCQ)  holds for multifunction $\multifC:\, \setD\rightrightarrows \hilbertH$ given by \eqref{M(p,v)} at $(\bar{p},\bar{x})$, $\bar{x}\in C(\bar{p})$, if there exists a neighbourhood $U(\bar{p})$ of $\bar{p}$ such that, for any index set $J$, $I_1\subset J\subset I_{\bar{p}}(\bar{x})$, for every $p\in U(\bar{p})$ the system of vectors $\{ g_i(p), i\in J  \}$ has constant rank. Precisely, for any $J$, $I_1\subset J\subset I_{\bar{p}}(\bar{x})$
	\begin{equation*}
		\rank(g_i(p),i\in J)=\rank(g_i(\bar{p}),i\in J)\quad \text{for all }p\in U(\bar{p}.)
	\end{equation*} 
\end{definition}
For more general constraint sets this definition has been introduced in \cite[Definition 1]{parametric_nonlinear_programming_Minchenko}. In \cite{Kruger2014}  several kinds of relations between constraint qualifications (for $C(\bar{p})$) has been established including RCRCQ and the classical Mangasarian Fromovitz Constraint Qualification (MFCQ). 

The following diagram provides the summary of the existing results concerning $R$-regularity, calmness, metrical subregularity, metric regularity of sets and multifunctions $\multifC(p)$. Let us note however that it also applies to more general forms of sets and multifunctions.

$$
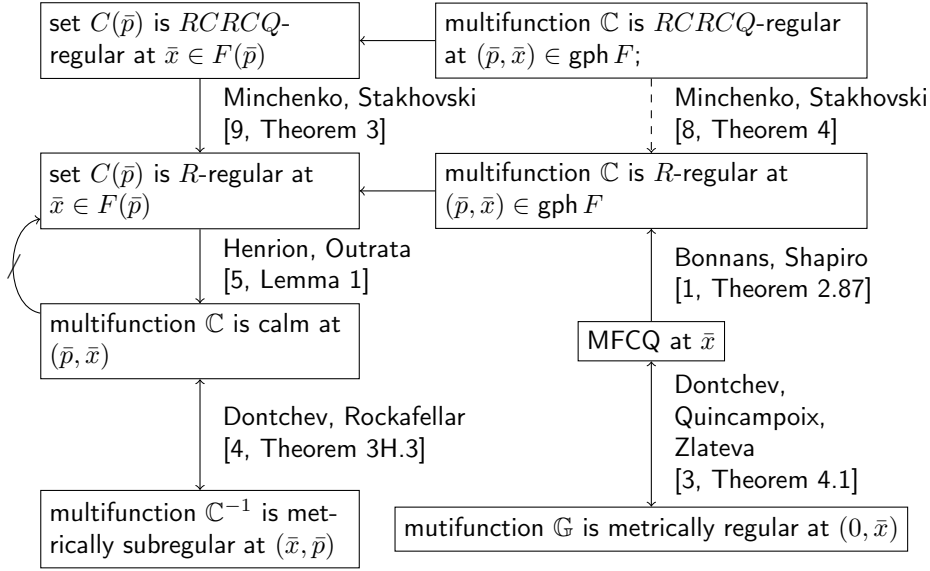
\begin{tikzpicture}[node distance=2cm, auto]\
\node[draw] at (0,0) (RCR-regular) [text width=4cm] {set $C(\bar{p})$ is $RCRCQ$-regular at $\bar{x}\in F(\bar{p})$};
\node[draw] at (0,-2) (R-regular) [text width=4cm] {set $C(\bar{p})$ is $R$-regular at $\bar{x}\in F(\bar{p})$};
\node[draw] at (0,-4) (calm) [text width=4cm] {multifunction $\multifC$ is calm at $(\bar{p},\bar{x})$};
\node[draw] at (0,-6.5) (calm-G) [text width=4cm] {multifunction $\multifC^{-1}$ is metrically subregular at $(\bar{x},\bar{p})$};
\node[draw] at (6,0) (RCR-regular-MF) [text width=5.5cm] {multifunction $\multifC$ is $RCRCQ$-regular at $(\bar{p},\bar{x})\in\gph F$;
	%	and $\bar{x}\in\liminf\limits_{p\rightarrow \bar{p}} F(p)$
};
\node[draw] at (6,-2) (R-regular-MF) [text width=5.5cm] {multifunction $\multifC$ is $R$-regular at $(\bar{p},\bar{x})\in\gph F$};
\node[draw] at (6,-4) (MFCQ)  {MFCQ at $\bar{x}$};
\node[draw] at (6,-6.5) (G-metric) {mutifunction ${\mathbb G}$ is metrically regular at $(0,\bar{x})$};
\draw[->] (R-regular) to node %[swap]
{\begin{tabular}{l}Henrion, Outrata\\ \cite[Lemma 1]{on_the_calmness_of_a_class-of_multifunctions}\end{tabular}} (calm);
\draw[<->] (calm) to node %[xshift=0.5cm,yshift=-0.3cm]
{	\begin{tabular}{l}
	Dontchev, Rockafellar\\ \cite[Theorem 3H.3]{implicit_functions_and_solution_mapping} \end{tabular}	} (calm-G);
\draw[<->] (G-metric) to node %[xshift=0.5cm,yshift=-0.3cm]
[swap] {\begin{tabular}{l}
	\text{Dontchev,} \\
	\text{Quincampoix,}\\
	\text{Zlateva}\\
	\cite[Theorem 4.1]{aubin_criterion_for_metric_regularity}
	\end{tabular}} (MFCQ);
\draw[->] (RCR-regular) to node {\begin{tabular}{l}Minchenko, Stakhovski\\ \cite[Theorem 3]{om_relaxed_constant_rank_regularity_condition}\end{tabular} } (R-regular);
\draw[->, dashed] (RCR-regular-MF) to node {\begin{tabular}{l}Minchenko, Stakhovski\\ \cite[Theorem 4]{parametric_nonlinear_programming_Minchenko}\end{tabular}} (R-regular-MF);
\draw[->] (RCR-regular-MF) to node {} (RCR-regular);
\draw[->] (R-regular-MF) to node {} (R-regular);
\draw[->]  (calm) to [out=170, in=190] node [yshift=0.cm]{$\not$} (R-regular);
\draw[->]  (MFCQ) to node [swap] {\begin{tabular}{l}Bonnans, Shapiro\\ \cite[Theorem 2.87]{Bonnans_Shapiro}\end{tabular}} (R-regular-MF);
\end{tikzpicture}$$
In the diagram multifunction ${\mathbb G}$ is defined as ${\mathbb G}=\bar{G}+K$, where $K=\{0\}^m\times \mathbb{R}_{+}^{n-m}$.
Implication given as dotted line under additional assumption has been proposed in \cite[Theorem 4]{parametric_nonlinear_programming_Minchenko}. However, as mentioned in Introduction the proof of \cite[Theorem 4]{parametric_nonlinear_programming_Minchenko} is incorrect. In the next section we present a counterexample to the proof of \cite[Lemma 3]{parametric_nonlinear_programming_Minchenko} and propose a new proof in our settings.

%\cite[Theorem 3]{om_relaxed_constant_rank_regularity_condition}
%\cite[Lemma 1]{on_the_calmness_of_a_class-of_multifunctions}
%\cite[Theorem 3H.3]{implicit_functions_and_solution_mapping}
%\cite[Theorem 4]{parametric_nonlinear_programming_Minchenko}
%\cite[Theorem 4.1]{aubin_criterion_for_metric_regularity}
%\cite[Theorem 2.87]{Bonnans_Shapiro}

\section{Main result}

We start with the proposition which relates  RCRCQ condition to the boundedness  (with respect to ${p},w$) of Lagrange multiplier set 
\begin{equation*}
	\Lambda_{w}^M(p,P_{C(p)}(w)):=\{ \lambda \in \Lambda_{w}(p,P_{C(p)}(w))\ |\ \sum_{i=1}^{n} |\lambda_i|\leq M\}  .
\end{equation*}

\begin{proposition}\label{theorem:RCRCQ_to_condition}
	Let multifunction $\multifC$ given by \eqref{M(p,v)} satisfy RCRCQ at $(\bar{x},\bar{p})\in \gph \multifC$. Assume that $\bar{x}\in \liminf\limits_{p\rightarrow \bar{p}} \multifC(p)$. Then there exist numbers $M>0$, $\delta>0$, $\delta_0>0$ such that 
	\begin{equation*}
		\Lambda_{w}^M(p,P_{C(p)}(w))\neq \emptyset\quad  \text{ for $p\in \bar{p}+\delta_0B$, $w\in \bar{x}+\delta B$, $w\notin C(p)$.}
	\end{equation*}	
	
\end{proposition}

The content of Proposition \ref{theorem:RCRCQ_to_condition} coincides with the content of \cite[Lemma 3]{parametric_nonlinear_programming_Minchenko}. The proof of Proposition \ref{theorem:RCRCQ_to_condition} we present below is essentially different from the proof of  Lemma 3 of 
\cite{parametric_nonlinear_programming_Minchenko}. The proof of \cite[Lemma 3]{parametric_nonlinear_programming_Minchenko} is incorrect which can be shown by the the following example.

\begin{example}%(Incorrectness of the proof of \cite[Lemma 3]{parametric_nonlinear_programming_Minchenko}  )\normalfont
	
	Let $\multifC:\ \mathbb{R}^2\rightarrow \mathbb{R}^2$ be defined as follows
	\begin{equation}
	\label{eq_C}
	C(p):=\left\{ x\in \mathbb{R}^2 \bigg| 
	\begin{array}{l}
	\langle x \ |\ (1,0)\rangle =0\\
	\langle x \ |\ (0,1)\rangle =0\\
	\langle x \ |\ p \rangle \leq0
	\end{array}
	\right\}
	\end{equation}
	and $\bar{p}=\bar{x}=(0,0)$. We have $C(p)=\{(0,0)\}$ for all $p=(r_1,r_2)\in \mathbb{R}^2$ and
	\begin{enumerate}
		\item RCRCQ holds for multifunction $\multifC$  at $z_0=((0,0),(0,0))\in \gph(\multifC)$,
		\item $(0,0)\in \liminf\limits_{p\rightarrow (0,0)} \multifC(p)$.
	\end{enumerate}
	%(Proof of \cite[Lemma 3]{parametric_nonlinear_programming_Minchenko}) 
	
	We have $g_1(p)=(1,0)$, $g_2(p)=(0,1)$, $g_3(p)=p$ for all $p\in \mathbb{R}^2$ and $G_1(p,x)=\langle x \ |\ (1,0)\rangle$, $G_2(p,x)=\langle x \ |\ (0,1)\rangle$, $G_3(p,x)=\langle x \ |\ p \rangle$ and the assumptions of \cite[Lemma 3]{parametric_nonlinear_programming_Minchenko} are satisfied.
	
	The proof of \cite[Lemma 3]{parametric_nonlinear_programming_Minchenko} relies on showing that for any sequences $p_k\rightarrow \bar{p},\ w_k\rightarrow \bar{x}$, $w_k\notin C(p_k)$ there exist 
	$$\lambda_k\in \Lambda_{w_k}^M(p_k,P_{C(p_k)}(w_k))\ \text{for some } M\geq 0  \text{ and all } k\in \mathbb{N}.
	$$ 
	
	Below we show that the way of choosing $\lambda_k$  which are to satisfy the above property is incorrect in general. More precisely, we show that for $\multifC$ defined by \eqref{eq_C}
	there are sequences $p_k\rightarrow \bar{p},w_k\rightarrow \bar{x}$ and $\lambda_k\in \Lambda_{w_k}(p_k,P_{C(p_k)}(w_k))$ chosen as in the proof of \cite[Lemma 3]{parametric_nonlinear_programming_Minchenko}
	with $\|\lambda_k\|\rightarrow +\infty$.
	
	Let $p_k=(\frac{1}{k^2},\frac{1}{k^2})\rightarrow (0,0)$, $w_k=(\frac{1}{k},\frac{2}{k})$. We have that $x_k=(0,0)=\Pi_{C(p_k)}(w_k)$ and in the notation of the proof of  \cite[Lemma 3]{parametric_nonlinear_programming_Minchenko}, $z_k=\left((\frac{1}{k^2},\frac{1}{k^2}),(0,0)\right)$.
	We have $I_{p_k}(x_k)=I^*=\{1,2,3\}$ and
	\begin{align*}
		&0=\frac{P_{C(p_k)(w_k)}-w_k}{\|P_{C(p_k)(w_k)}-w_k\|}+\sum_{i\in I_{p_k}(x_k)}\lambda_i g_i(p_k)\\
		&\Leftrightarrow\quad \left(\frac{1}{\sqrt{5}},\frac{2}{\sqrt{5}}\right)=\lambda_1 (1,0)+\lambda_2 (0,1)+\lambda_3 \left(\frac{1}{k^2},\frac{1}{k^2}\right).
	\end{align*} 	
	There exists a maximal linearly independent subfamily $\{ g_i(p_k),\ i\in \{2,3\}   \}$ in the family  $\{ g_i(p_k),\ i\in \{1,2,3\}   \}$ such that  $(0,\frac{1}{\sqrt{5}},\frac{k^2}{\sqrt{5}})\in \Lambda_{(\frac{1}{k^2} ,\frac{2}{k^2})}((\frac{1}{k^2},\frac{1}{k^2}),(0,0))$ for all $k\in \mathbb{N}$.
	
	In the notation of the proof of  \cite[Lemma 3]{parametric_nonlinear_programming_Minchenko} 
	we have $J(z_k)=J^0=\{2,3\}$. RCRCQ at the point $z_0$ implies that
	\begin{equation*}
		2=\rank \{ g_i(\bar{p}),\ i\in \{1,2,3\} \}=\rank \{ g_i(p),\ i\in \{1,2,3\} \}
	\end{equation*}
	for all points $z\in \mathbb{R}^2$. Moreover for all $z_k$, $k=1,2,\dots$ we have
	\begin{equation*}
		2=\rank \{ g_i(p_k),\ i\in \{1,2,3\} \}=\rank \{ g_i(p),\ i\in \{2,3\} \}.
	\end{equation*}
	Observe  that $\rank\{g_1(p),g_2(p) \}=2$ for all $p\in U((0,0))$. Hence, in the notation of the proof of  \cite[Lemma 3]{parametric_nonlinear_programming_Minchenko},  $J^{00}=\{1,2\}$ and the  function $\Phi$ takes the form
	\begin{equation*}
		\begin{array}{l}
			G_1(p,x)=\Phi(G_1(p,x),G_2(p,x)),\\
			G_2(p,x)=\Phi(G_1(p,x),G_2(p,x)),\\
			G_3(p,x)=\Phi(G_1(p,x),G_2(p,x))=r_1 G_1(p,x)+r_2 G_2(p,x)\\
			\text{(since $\langle x \ |\ p\rangle = \langle x \ |\  \langle p \ |\ (1,0)\rangle  \cdot (1,0)\rangle + \langle x \ |\  \langle p \ |\ (0,1)\rangle\cdot (0,1)\rangle $)}.
		\end{array}
	\end{equation*}
	On the other hand,
	\begin{align*}
		g_3(p_k)&=\nabla_x \Phi(G_1(p_k,x_k),G_2(p_k,x_k))=\frac{1}{k} (1,0)+\frac{1}{k} (0,1),\\
		g_3(\bar{p})&=\nabla_x \Phi(G_1(\bar{p},\bar{x}),G_2(\bar{p},\bar{x}))=0\cdot (1,0)+0\cdot (0,1),
	\end{align*}
	and vectors $g_2((0,0))=(1,0)$, $g_3((0,0))=(0,0)$ are linearly dependent. Moreover, $\|(0,\frac{1}{\sqrt{5}},\frac{k^2}{\sqrt{5}})\|=\sqrt{\frac{1}{5}+\frac{k^4}{5}}\rightarrow +\infty$ and
	\begin{align*}
		(0,0)&=\lim_{k\rightarrow +\infty} \frac{1}{\sqrt{\frac{1}{5}+\frac{k^4}{5}}} (\frac{1}{k},\frac{2}{k})=\lim_{k\rightarrow +\infty}\frac{\sqrt{5}}{k^2\sqrt{\frac{1}{k^4}+1}}(\frac{1}{k},\frac{2}{k})\\
		&=\lim_{k\rightarrow +\infty} \frac{\sqrt{5}}{k^3\sqrt{\frac{1}{k^4}+1}}(0,1)+\frac{\sqrt{5}}{k\sqrt{\frac{1}{k^4}+1}}(\frac{1}{k^2},\frac{1}{k^2})=0(0,1)+0(0,0).
	\end{align*}
	The example shows that the construction proposed in the proof of \cite[Lemma 3]{parametric_nonlinear_programming_Minchenko} may lead to the contradiction of the conclusion.	The reason is that in the proof of \cite[Lemma 3]{parametric_nonlinear_programming_Minchenko} the set $J^{0}$ is chosen in an incorrect way and function $\Phi$ does not depend on $p$ directly.
	%	In view of above there is no contradictory with the assumption of contradiction in the proof.
\end{example}

\begin{proof}[Proof of Proposition \ref{theorem:RCRCQ_to_condition}]
	On the  contrary suppose, that there exist sequences $p_k\rightarrow \bar{p}$, $w_k\rightarrow \bar{x}$ such that $w_k\notin C(p_k)$ and
	\begin{equation}\label{assumption:unbounded}
	\dist (0,\Lambda_{w_k}(p_k,P_{C(p_k)}(w_k))\rightarrow+\infty.
	\end{equation} 
	
	Due to the fact that $\bar{x}\in \liminf\limits_{p\rightarrow \bar{p}} \multifC(p),$ we may assume without loss of generality that $C(p_k)\neq \emptyset$ for each $p_k$, and there exists $\hat{x}_k\in C(p_k)$ such that $\hat{x}_k\rightarrow\bar{x}$.

	RCRCQ   at $(\bar{p},\bar{x})$ implies that RCRCQ holds also at all the points near the point $(\bar{p},\bar{x})$.  Without loss of generality one may assume that RCRCQ holds at all $(p_k,P_{C(p_k)}(w_k))$, $k\in \mathbb{N}$. Consequently, $\Lambda_{w_k}(p_k,P_{C(p_k)}(w_k))\neq \emptyset$ for all $k=1,2,\dots$. 
	
	Passing to subsequences, if necessary, we may assume that $(p_k,w_k)\in V(\bar{p},\bar{w})$, where by RCRCQ, $V(\bar{p},\bar{w})$ is such that for any $J$, $I_1\subset J\subset I_1\cup I_2$
	\begin{equation}\label{conditon:RCRCQ}
	\rank\{ g_i(p_k),\ i\in J\} =\rank\{g_i(\bar{p}),\ i\in J\}.
	\end{equation}
	By Theorem \ref{theorem:representation},
	\begin{equation}\label{eq:representation_lambda}
	w_k-P_{C(p_k)}(w_k)=\sum_{i\in I_{p_k}(P_{C(p_k)}(w_k))} \hat{\lambda}_i^k g_i(p_k),\quad k=1,\dots
	\end{equation}
	where $\hat{\lambda}_i^k\geq 0$, $i\in I_2\cap I_{p_k}(P_{C(p_k)}(w_k))$. 	Recall that $I_{p}(P_{C(p)}(w)):=\{ i\in I_1\cup I_2\ |\ \langle P_{C(p)}(w)\ |\ g_i(p)\rangle - f_i(p)=0  \}$ and $\hat{\lambda}_i^k$, $i\in I_2\cap I_{p_k}(P_{C(p_k)}(w_k))$ are related to the set $\Lambda_{w_k}(p_k,P_{C(p_k)}(w_k))$ via equivalence \eqref{equivalence:relations_lambda}. Then \eqref{eq:representation_lambda} takes the form
	\begin{align}
		\begin{aligned}\label{eq:representation_lambda1}
			& w_k-P_{C(p_k)}(w_k)=\sum_{I_1} \hat{\lambda}_i^k g_i(p_k)+\sum_{I_{p_k}(P_{C(p_k)}(w_k)\setminus I_1} \hat{\lambda}_i^k g_i(p_k),\\
			& \hat{\lambda}_i^k\geq 0, i\in I_{p_k}(P_{C(p_k)}(w_k)\setminus I_1\quad k=1,\dots
		\end{aligned}
	\end{align}
	By Lemma \ref{lemma:linearly_independent_equalities}, there exists $I_1^0\subset I_1$, $I_2^0(w_k,p_k)\subset I_2$,  and $\tilde{\lambda}_i(w_k,p_k)\in \mathbb{R}$, $i\in I_1^0$, $\tilde{\lambda}_i(w_k,p_k)>0$, $i\in I_2^0(w_k,p_k)$ such that	
	\begin{equation}\label{eq:representation_sequence}
	w_k-P_{C(p_k)}(w_k)=\sum_{i \in I_1^0} \tilde{\lambda}_i(w_k,p_k) g_i(p)+\sum_{I_2^0(w_k,p_k)} \tilde{\lambda}_i(w_k,p_k) g_i(p_k),
	\end{equation}
	where $g_i(p_k)$, $i\in I_1^0\cup I_2^0(w_k,p_k)$ are linearly independent.
	
	Passing to a subsequence, if necessary, we may assume that for all $k\in \mathbb{N}$, $ I_2^0(w_k,p_k)$ is a fixed set, i.e., $ I_2^0(w_k,p_k)=I_2^0$.
	
	By RCRCQ, there exists $k_0$ such that for all $k\geq k_0$
	\begin{equation*}
		\rank\{ g_i(p_k) , \ i\in I_1^0\cup I_2^0 \}=\rank\{ g_i(\bar{p}) , \ i\in I_1^0\cup I_2^0 \}.
	\end{equation*}
	Put $\lambda_i^k=\frac{\tilde{\lambda}_i^k}{\|w_k-P_{C(p_k)}(w_k)\|}$. 	
	For every $k\geq k_0$ we have  $\lambda_k(w_k,p_k)\in \Lambda_{w_k}(p_k,P_{C(p_k)}(w_k))$ and by  \eqref{assumption:unbounded}, $\|{\lambda}^k(w_k,p_k)\|\rightarrow+\infty $. Without loss of generality we may assume that ${\lambda}^k(w_k,p_k)
	\|{\lambda}^k(w_k,p_k)\|^{-1}\rightarrow\bar{\lambda}$. Then by \eqref{eq:representation_sequence} we obtain
	\begin{equation*}
		0=\sum_{i\in I_1^0\cup I_2^0} \bar{\lambda}_i g_i(\bar{p}),\ \bar{\lambda}_i\geq 0,\ i\in I_2^0, 
	\end{equation*}
	where $\|\bar{\lambda}\|=1$. This contradicts the fact that $g_i(\bar{p})$, $i\in I_1^0\cup I_2^0$ are linearly independent.
\end{proof}

In the next proposition we relate the boundedness of the Lagrange multiplier set $\Lambda_{w}^M(p,P_{C(p)}(w))$ to the $R$-regularity of $\multifC$ at $(\bar{p},\bar{x})$. For  sets $C(p)$ given as solution sets to parametric systems of nonlinear equations and inequalities in finite dimensional spaces this fact has been already proved
in \cite[Theorem 2]{parametric_nonlinear_programming_Minchenko}.
The proof we give below is based on the proof of  Theorem 2 of \cite{parametric_nonlinear_programming_Minchenko}.

\begin{proposition}\label{theorem:condition_to_R-regular}
	Let $\bar{p}\in \setD$, $\bar{x}\in C(\bar{p})$ and
	$\bar{x}\in \liminf\limits_{p\rightarrow\bar{p}} \multifC(p)$. Assume that there exist numbers $M>0$, $\delta_1>0$, $\delta_2>0$ such that 
	$$\Lambda_{w}^M(p,P_{C(p)}(w)):=\{ \lambda \in \Lambda_{w}(p,P_{C(p)}(v))\ |\ \sum_{i=1}^{n} |\lambda_i|\leq M\}\neq \emptyset 
	$$ for all $p\in (\bar{p}+\delta_1B)\cap S$ and for all $w\in (\bar{x}+\delta_2 B)$, $w\notin C(p)$. Then the multifuction $\multifC$ is $R$-regular at  $(\bar{x},\bar{p})$.
\end{proposition}
\begin{proof}
	Since $\bar{x}\in \liminf\limits_{p\rightarrow \bar{p}} \multifC(p)$ one can find $\delta_3>0$ such that $C(p)\cap \{\bar{x}+4^{-1} \delta_3 B  \}\neq \emptyset$ for all $p\in \bar{p}+\delta_3 B$. Let $p\in \bar{p}+2^{-1}\delta_3B$, $w\in \bar{x}+4^{-1}\delta_3B$. If $w\in C(p)$ then $\dist(w,C(p))=0$.
	
	Let $w\notin C(p)$ and $w\in \bar{x}+4^{-1}\delta_3 B$. Since $C(p)\cap \{\bar{x}+4^{-1}\delta_3B\}\neq \emptyset$ there exists $x_1\in C(p)\cap \{ \bar{x}+4^{-1}\delta_3 B\}$. Then
	\begin{equation*}
		\| P_{C(p)}(w)-w\|\leq \| w- x_1\|\leq \|w-\bar{x}\|+\|x_1-\bar{x}\|<2^{-1}\delta_3
	\end{equation*}
	It follows that $P_{C(p)}(w)\in \bar{x}+\delta_3B$.
	Let
	\begin{equation*}
		\lambda\in \bigg\{ \lambda\in \mathbb{R}^n \ |\ \sum_{i=1}^{n}|\lambda_i| \leq 2M \bigg\}.
	\end{equation*}
	Introduce a function
	\begin{equation*}
		h(p,x)=h(p,x,w,\lambda,P_{C(p)}(w))=\dfrac{\langle x-w \ |\ x-P_{C(p)}(w)\rangle  }{\|P_{C(p)}(w)-w\| }+\sum_{i=1}^n \lambda_i G_i(p,x)
	\end{equation*}
	The function $h(p,x)$ is convex with respect to $x$ on $\hilbertH$.

	Let $\lambda\in \Lambda_w^M(p,P_{C(p)}(w))$, $p\in \bar{p}+2^{-1}\delta_3B$, $w\in \bar{x}+4^{-1}\delta_3B$ such that $w\notin C(p)$. Since $\Lambda_w^*(x,P_{C(p)}(w))=2\Lambda_w(x,P_{C(p)}w))\neq \emptyset$ by \eqref{set:2lambda} we have $\lambda^*:=2\lambda\in \Lambda_w^*(x,P_{C(p)}(w)) $.
	
	The equality $\nabla_x L_x^*(p,P_{C(p)}(w),\lambda^*)=0$ can be written in the form
	\begin{equation*}
		\frac{w-P_{C(p)}(w)}{\|w-P_{C(p)}(w)\|}=\frac{P_{C(p)}(w)-v}{\|P_{C(p)}(w)-w\|}+\sum_{i=1}^n \lambda_i^* g_i(p),
	\end{equation*}
	where the right side coincides with the gradient $\nabla_x h(p,x)$ of the function 
	\begin{equation*}
		h(p,x)=h(p,x,w,\lambda^*,P_{C(p)}(w))=\dfrac{\langle x-w \ |\ x-P_{C(p)}(w)\rangle  }{\|P_{C(p)}(w)-w\| }+\sum_{i=1}^n \lambda_i^* G_i(p,x)
	\end{equation*}
	at the point $y=P_{C(p)}(w)$.

	Since
	\begin{equation*}
		\langle \nabla_x h(p, P_{C(p)}(w))\ |\ w- P_{C(p)}(w) \rangle \leq h(p,w)-h(p,P_{C(p)}(w))
	\end{equation*}
	due to convexity of the function $h(p,x)$ with
	\begin{equation*}
		\lambda^*\in\Lambda_w^*(x,P_{C(p)}(w))\cap \bigg\{ \lambda\in \mathbb{R}^n \ |\ \sum_{i=1}^n |\lambda_i|\leq 2M  \bigg\},
	\end{equation*}
	from the last inequality it follows that
	\begin{align*}
		\|w&-P_{C(p)}(w)\|=\frac{\langle w- P_{C(p)}(w)\ |\ w- P_{C(p)}(w)\rangle}{\|P_{C(p)}(w)-w\|}\\
		&=\bigg\langle \frac{P_{C(p)}(w)-w}{\|P_{C(p)}(w)-w\|}+\sum_{i=1}^{n}\lambda_i^* g_i(p)\ |\ w- P_{C(p)}(w)\bigg\rangle\\
		&\leq \frac{\langle w-w \ |\ w- P_{C(p)}(w)\rangle }{\|P_{C(p)}(w)-w\|}+\sum_{i=1}^n \lambda_i^* G_i(p,w)\\
		&-\frac{\langle P_{C(p)}(w)- w \ |\ P_{C(p)}(w)-P_{C(p)}(w)\rangle}{\|P_{C(p)}(w)-w\|}-\sum_{i=1}^n \lambda_i^* G_i(p,P_{C(p)}(w))\\
		&=\sum_{i=1}^n \lambda_i^* (G_i(p,w)-G_i(p,P_{C(p)}(w)))=\sum_{i=1}^n \lambda_i^* G_i(p,w)=2\sum_{i=1}^n \lambda_i G_i(p,w).
	\end{align*}
	This inequality implies
	\begin{align*}
		&\dist (w,C(p))=\|w-P_{C(p)}(w)\|\leq 2 \|\lambda\|_1 \max\{0, G_i(p,w),\ i \in I_2,\ |G_i(p,w)|, i\in I_1 \}\\
		&\leq 2M\max\{0, G_i(p,w),\ i \in I_2,\ |G_i(p,w)|, i\in I_1 \}.
	\end{align*}
\end{proof}

Now we show that if the multifunction $\multifC$ is $R$-regular at $(\bar{p},\bar{x})$ then $\multifC$ is Lipschitz like at $(\bar{p},\bar{x})$.
\begin{proposition}\label{theorem:r-regular_to_lipschitz-like}
	Let $\hilbertH,\ \hilbertG$ be a Hilbert spaces and $f_i:\ \setD\rightarrow \mathbb{R}$, $g_i:\ \setD\rightarrow \hilbertH$ are Lipschitz on $\setD\subset \hilbertG$. If the set-valued mapping $\multifC:\ \setD \rightrightarrows \hilbertH$ given by \eqref{M(p,v)} is R-regular at $(\bar{p},\bar{x})$, $\bar{p}\in \setD$, $\bar{x}\in C(\bar{p})$ then $\multifC$ is Lipschitz-like at $(\bar{p},\bar{x})$ 
\end{proposition}
\begin{proof}
	By the $R$-regularity of $\multifC$ there exists a constant $\alpha$ and a neighbourhood  $U(\bar{p})$ and a neighbourhood $V(\bar{x})$ such that
	\begin{equation*}
		\text{dist}(x,C(p))\leq \alpha \max\{ 0, |\langle x \ |\  g_i(p)\rangle -f_i(p)|, i\in I_1,\ \langle x \ |\  g_i(p)\rangle -f_i(p),\ i \in I_2  \}
	\end{equation*}
	for all $(p,x)$ in neighbourhood  $U(\bar{p})\times V(\bar{x})$. Let $(p_1,x_1)$, $x_1\in C(p_1)$ in neighbourhood  $U(\bar{p})\times V(\bar{x})$ and $p_2\in U(\bar{p})$. Since $C(p_2)$ is closed and convex there exists $x_2\in C(p_2)$ such that $\text{dist}(x_1,C(p_2))=\|x_1-x_2\|$. Then by $R$-regularity
	\begin{align*}
		&\text{dist}(x_1,C(p_2))=\|x_1-x_2\|\\
		&\leq \alpha
		\max\left\{0,\max_{i\in I_1} |\langle x_1 \ |\  g_i(p_2)\rangle -f_i(p_2)|, \max_{i\in I_2}\  \langle x_1 \ |\  g_i(p_2)\rangle -f_i(p_2) \right\}\\
		&\leq \alpha
		\max\left\{0,\max_{i\in I_1} |\langle x_1 \ |\  g_i(p_2)\rangle -f_i(p_2)-(|\langle x_1 \ |\  g_i(p_1)\rangle -f_i(p_1))|,\right.\\
		& \left.\qquad\max_{i\in I_2}  \langle x_1 \ |\  g_i(p_2)\rangle -f_i(p_2)-(|\langle x_1 \ |\  g_i(p_1)\rangle -f_i(p_1)) \right\}\\
		&=\alpha
		\max\left\{0,\max_{i\in I_1} |\langle x_1 \ |\  g_i(p_2)-g_i(p_1)\rangle -(f_i(p_2)-f_i(p_1))|,\right.\\
		&\left.\qquad\max_{i\in I_2}  \langle x_1 \ |\  g_i(p_2)-g_i(p_1)\rangle -(f_i(p_2)-f_i(p_1)) \right\}\\
		&\leq \alpha
		\max\left\{\max_{i\in I_1} \|x_1\|\|g_i(p_2)-g_i(p_1)\|+ \|f_i(p_2)-f_i(p_1)\|,\right.\\
		&\left.\qquad \max_{i\in I_2}  \|x_1\| \|g_i(p_2)-g_i(p_1)\| +\|f_i(p_2)-f_i(p_1)\| \right\}\\
		&=\alpha
		\max_{i\in I_1\cup I_2} \|x_1\|\|g_i(p_2)-g_i(p_1)\|+ \|f_i(p_2)-f_i(p_1)\|\\
		&\leq \alpha \max_{i\in I_1\cup I_2} (\|x_1\|\ell_{g_i}+\ell_{f_i})\|p_1-p_2\|,
	\end{align*}
	hence $\multifC$ is Lipschitz-like at $(\bar{x},\bar{p})$.
\end{proof}
%In finite dimensional setting Proposition \ref{theorem:condition_to_R-regular} has been shown in \cite[Theorem 2]{parametric_nonlinear_programming_Minchenko} for parametric sets given as systems of nonlinear equalities and inequalities.

The following theorem is our main result.
%We summarize the above results in the  following Theorem. 
\begin{theorem}
	Let multifunction $\multifC:\ \setD\rightrightarrows \hilbertH$ given by \eqref{M(p,v)} satisfy RCRCQ at $(\bar{x},\bar{p})\in \gph \multifC$. Assume that $\bar{x}\in \liminf\limits_{p\rightarrow \bar{p}} \multifC(p)$. Then  $\multifC$ is Lipschitz-like at $(\bar{p},\bar{x})$ 
\end{theorem}
\begin{proof}
	The proof follows immediately from Proposition \ref{theorem:RCRCQ_to_condition}, Proposition \ref{theorem:condition_to_R-regular}, Proposition \ref{theorem:r-regular_to_lipschitz-like}.	
\end{proof}

\section{Conclusions}
In this paper we used RCRCQ to investigate Lipschitz-likeness of set valued mapping $\multifC$ given by \eqref{M(p,v)}. 
In many existing papers (e.g. \cite{Bonnans_Shapiro,implicit_functions_and_solution_mapping,aubin_criterion_for_metric_regularity,on_the_calmness_of_a_class-of_multifunctions}) the continuity properties of set-valued mappings are related to the Mangasarian-Fromovitz constraint qualification MFCQ. In general, there is no direct relationship between RCRCQ and MFCQ (see \cite{Kruger2014}). It depends upon the problem considered which of the two constraint qualifications is more useful.\\[0.8cm] \mbox{\ }

\section{Appendix}
\begin{lemma}\label{lemma:linear_independent}
	Let $J=\{1,\dots,k\}$. Let $g_i:\ \hilbertG\rightarrow \hilbertH$, $i\in J$ be continuous operators and let $\bar{p}$ be such that $g_i(\bar{p})$, $i\in J$ are linearly independent. Then there exists a neighbourhood $U(\bar{p})$ such that for all $p\in U(\bar{p})$, $g_i(p)$, $i\in J$ are linearly independent.
\end{lemma}
\begin{proof}
	The fact that $g_i(\bar{p})$, $i\in J$ are linearly independent is equivalent to the fact that the Gram determinant of $g_i(\bar{p})$, $i\in J$ is nonzero (see for example \cite[Lemma 7.5]{deutsch2001best}), i.e
	\begin{align*}
		&\text{Gram}(g_1(\bar{p}),\dots,g_k(\bar{p})):=\\
		&\left|\begin{array}{cccc}
			\langle g_1(\bar{p}) \ |\ g_1(\bar{p}) \rangle & \langle g_1(\bar{p}) \ |\ g_2(\bar{p}) \rangle & \dots & \langle g_1(\bar{p}) \ |\ g_k(\bar{p})\rangle\\
			\langle g_2(\bar{p}) \ |\ g_1(\bar{p}) \rangle & \langle g_2(\bar{p}) \ |\ g_2(\bar{p}) \rangle & \dots & \langle g_2(\bar{p}) \ |\ g_k(\bar{p})\rangle\\
			\dots\\
			\langle g_k(\bar{p}) \ |\ g_1(\bar{p}) \rangle & \langle g_k(\bar{p}) \ |\ g_2(\bar{p}) \rangle & \dots & \langle g_k(\bar{p}) \ |\ g_k(\bar{p})\rangle
		\end{array}\right|\neq 0.
	\end{align*}
	For any $p$ let
	\begin{align*}%\label{Gram_determinant}
		&{\mathcal F}(p):=\text{Gram}(g_1(p),\dots,g_k(p)):=\\
		&\left|\begin{array}{cccc}
			\langle g_1(p) \ |\ g_1(p) \rangle & \langle g_1(p) \ |\ g_2(p) \rangle & \dots & \langle g_1(p) \ |\ g_k(p)\rangle\\
			\langle g_2(p) \ |\ g_1(p) \rangle & \langle g_2(p) \ |\ g_2(p) \rangle & \dots & \langle g_2(p) \ |\ g_k(p)\rangle\\
			\dots\\
			\langle g_k(p) \ |\ g_1(p) \rangle & \langle g_k(p) \ |\ g_2(p) \rangle & \dots & \langle g_k(p) \ |\ g_k(p)\rangle
		\end{array}\right|
	\end{align*}
	Since inner product is a continuous function of arguments and ${\mathcal F}:\ \hilbertG\rightarrow \mathbb{R}$ is a  combination of continuous functions, there exists a neighbourhood $U(\bar{p})$ such that ${\mathcal F}(p)\neq 0$ for all $p\in U(\bar{p})$. Hence, for all $p\in U(\bar{p})$ vectors $g_i(p)$, $i\in J$ are linearly independent.
\end{proof}
\begin{proposition}\label{propostion:I_1-linearly_independent}
	Let $\bar{p}\in \setD$. Assume that RCRCQ holds at $\bar{p}$ for multifunction $\multifC$ given by \eqref{M(p,v)} and $C(p)\neq \emptyset$ for $p\in U_0(\bar{p})$. Then there exists a neighbourhood $U(\bar{p})$ such that for all $p\in U(\bar{p})$
	\begin{align*}
		&\{ x\ |\ \langle x \ |\ g_i(p)\rangle = f_i(p),\ i\in I_1,\ \langle x \ |\ g_i(p)\rangle \leq  f_i(p),\ i\in I_2 \}\\
		&=\{ x\ |\ \langle x \ |\ g_i(p)\rangle = f_i(p),\ i\in I_1^\prime,\ \langle x \ |\ g_i(p)\rangle \leq  f_i(p),\ i\in I_2 \},
	\end{align*}
	where $I_1^\prime\subset I_1$ and $g_i(p)$, $i\in I_1^\prime$ are linearly independent. 
\end{proposition}
\begin{proof}
	It is enough to consider the case $g_i(\bar{p})$, $i\in I_1$ are linearly dependent. By RCRCQ there exists a neighbourhood $U_1(\bar{p})$ such that for all $p\in U(\bar{p})$
	\begin{equation*}
		\rank\{  g_i(\bar{p}),\ i\in I_1 \}= \rank\{  g_i(p),\ i\in I_1 \}=\alpha.
	\end{equation*}
	Let $I_1^\prime$ be such that $|I_1^\prime|=\alpha$ and $g_i(\bar{p})$, $i\in I_1^\prime$ are linearly independent. By Lemma \ref{lemma:linear_independent}, there exists a neighbourhood $U_2(\bar{p})$ such that for all $p\in U_2(\bar{p})$, $g_i(p)$, $i\in I_1^\prime$ are linearly independent. Let $p\in U_0(\bar{p})\cap U_1(\bar{p})\cap U_2(\bar{p})$  $x$ be such that
	\begin{equation}
	\langle x \ |\ g_i(p)\rangle = f_i(p),\ i\in I_1,\ \langle x \ |\ g_i(p)\rangle \leq  f_i(p),\ i\in I_2.
	\end{equation} 
	Since $\rank\{  g_i(p),\ i\in I_1 \}=|I_1^\prime|$, $C(p)\neq \emptyset$ and $g_i(p)$, $i\in I_1^\prime$ are linearly independent we have
	\begin{align*}
		&\langle x \ |\ g_i(p)\rangle = f_i(p),\ i\in I_1\\
		& \Longleftrightarrow \langle x \ |\ g_i(p)\rangle = f_i(p),\ i\in I_1^\prime\ \wedge\ \langle x \ |\ g_i(p)\rangle = f_i(p),\ i\in I_1\setminus I_1^\prime\\
		&\Longleftrightarrow  \langle x \ |\ g_i(p)\rangle = f_i(p),\ i\in I_1^\prime\ \wedge\ \langle x \ |\ \sum_{j\in I_1^\prime} \alpha_j^i  g_j(p)\rangle =  f_i(p),\ i\in I_1\setminus I_1^\prime\\
		&\Longleftrightarrow  \langle x \ |\ g_i(p)\rangle = f_i(p),\ i\in I_1^\prime\ \wedge\ \sum_{j\in I_1^\prime}\alpha_j^i\langle x \ |\    g_j(p)\rangle = \sum_{j\in I_1^\prime} \alpha_j^i f_j(p),\ i\in I_1\setminus I_1^\prime\\
		&\Longleftrightarrow  \langle x \ |\ g_i(p)\rangle = f_i(p),\ i\in I_1^\prime,
	\end{align*}
	where $g_i(p)=\sum_{j\in I_1^\prime} \alpha_j^i  g_j(p)$, $f_i(p)=\sum_{j\in I_1^\prime} \alpha_j^i  f_j(p)$, $i\in I_1\setminus I_1^\prime$ and $\alpha_j^i\in \mathbb{R}$, $j\in I_1^\prime$, $i\in I_1\setminus I_1^\prime$, not all $\alpha_j^i$, $j\in I_1^\prime$ equal to zero for any $i\in I_1\setminus I_1^\prime$.
\end{proof}
\begin{lemma}\label{lemma:positive_combination2}
	Let $x=\sum_{i\in J_1} \lambda_i a_i + \sum_{i\in J_2} \lambda_i a_i$, $J_1\cap J_2=\emptyset$, $J_1,J_2$ finite sets, $\lambda_i\in \mathbb{R}$, $i\in J_1$, $\lambda_i\geq 0$, $i\in J_2$ and $a_i$, $i\in J_1\cup J_2$ are non-zero vectors. Assume that $a_i$, $i\in J_1$ are linearly independent. Then there exists $J_2^\prime\subset J_2$ and $\lambda_i^\prime$, $i\in J_1\cup J_2^\prime$, $\lambda_i^\prime\in \mathbb{R}$, $i\in J_1$, $\lambda_i^\prime>0$, $i\in J_2^\prime$ such that
	\begin{equation*}
		\sum_{i\in J_1} \lambda_i a_i + \sum_{i\in J_2} \lambda_i a_i=\sum_{i\in J_1} \lambda_i^\prime a_i + \sum_{i\in J_2^\prime} \lambda_i^\prime a_i
	\end{equation*} 
	and $a_i$, $i\in J_1\cup J_2^\prime$ are linearly independent.
\end{lemma}
\begin{proof}
	Without loss of generality we may assume that $\lambda_i>0$, $i\in J_2$. If $a_i$, $i\in J_1\cup J_2$ are linearly independent, then the assertion is obvious. Suppose that $a_i$, $i\in J_1\cup J_2$ are linearly dependent. Then there exists $\hat{J_1}\subset J_1$ and $\hat{J}_2\subset J_2$, $\hat{J}_2\neq \emptyset$ such that
	\begin{equation}\label{eq:vector_dependent}
	\sum_{i\in \hat{J}_1} \beta_i a_i +\sum_{i\in \hat{J}_2} \beta_i a_i=0,\quad \rank\{a_i,\ i\in \hat{J}_1\cup \hat{J}_2  \}= |\hat{J}_1\cup \hat{J}_2|-1
	\end{equation}
	for some $\beta_i\neq 0$, $i\in \hat{J}_1\cup \hat{J}_2$. Then by multiplying  both sides of equality equality \eqref{eq:vector_dependent} by $\frac{\lambda_k}{\beta_k}$, $k\in \hat{J}_2$ we get
	\begin{equation*}
		\sum_{i\in \hat{J}_1} \frac{\lambda_k}{\beta_k}\beta_i a_i +\sum_{i\in \hat{J}_2}\frac{\lambda_k}{\beta_k} \beta_i a_i=0.
	\end{equation*}
	Therefore for any $k\in \hat{J}_2$ we have
	\begin{align*}
		x=&\sum_{i\in J_1} \lambda_i a_i + \sum_{i\in J_2} \lambda_i a_i-\sum_{i\in \hat{J}_1} \frac{\lambda_k}{\beta_k}\beta_i a_i -\sum_{i\in \hat{J}_2}\frac{\lambda_k}{\beta_k} \beta_i a_i\\
		&=\sum_{i\in J_1\setminus J_1^\prime} \lambda_i a_i + \sum_{i\in J_2\setminus J_2^\prime} \lambda_i a_i  + \sum_{i\in \hat{J}_1} (\lambda_i - \frac{\lambda_k}{\beta_k}\beta_i)a_i+\sum_{i\in \hat{J}_2\setminus\{k\}} (\lambda_i - \frac{\lambda_k}{\beta_k}\beta_i)a_i.
	\end{align*}
	We will show that there exists $k\in \hat{J}_2$ such that for any $i\in \hat{J}_2\setminus\{k\}$ we have
	\begin{equation*}
		\lambda_i - \frac{\lambda_k}{\beta_k}\beta_i\geq 0.
	\end{equation*}
	Suppose by contrary that for all $k\in \hat{J}_2$ there exists $i_k\in \hat{J}_2\setminus\{k\}$ such that
	\begin{equation*}
		\lambda_{i_k} < \frac{\beta_{i_k}}{\beta_k}\lambda_k.
	\end{equation*}
	Let us note that fact $\lambda_i>0$ for all $i\in \hat{J}_2$ implies that for all $k\in \hat{J}_2$ we have
	\begin{equation*}
		\frac{\beta_{i_k}}{\beta_k}> \frac{\lambda_{i_k}}{\lambda_k} >0.
	\end{equation*}
	Then there exist real numbers $\lambda_{i_1},\dots,\lambda_{i_q}$, where $i_1,\dots,i_q\subset \hat{J}_2$, $q\leq |\hat{J}_2|$ and
	\begin{equation*}
		\lambda_{i_1} < \frac{\beta_{i_1}}{\beta_{i_2}}\lambda_{i_2},\ \dots,\ \lambda_{i_{q-1}} < \frac{\beta_{i_{q-1}}}{\beta_{i_q}}\lambda_{i_q},\ \lambda_{i_q}< \frac{\beta_{i_{q}}}{\beta_{i_1}}\lambda_{i_1}.
	\end{equation*}
	However, this implies that
	\begin{equation*}
		\lambda_{i_1}<\frac{\beta_{i_1}}{\beta_{i_2}}\lambda_{i_2}<\frac{\beta_{i_1}}{\beta_{i_2}} \frac{\beta_{i_2}}{\beta_{i_3}} \lambda_{i_3}=\frac{\beta_{i_1}}{\beta_{i_3}}\lambda_{i_3}<\dots<\frac{\beta_{i_1}}{\beta_{i_q}}\lambda_{i_q}<\frac{\beta_{i_1}}{\beta_{i_q}}\frac{\beta_{i_q}}{\beta_{i_1}}\lambda_{i_1}=\lambda_{i_1},
	\end{equation*}
	which leads to a contradiction. 
	Hence, we can represent $x$ as 
	\begin{equation*}
		x=\sum_{i\in J_1} \lambda_i^\prime a_i + \sum_{i\in J_2^\prime} \lambda_i^\prime a_i
	\end{equation*} 
	where $J_2^\prime\subset J_2$, $\lambda_i^\prime>0$, $i\in J_2^\prime$ and $a_i$, $i\in J_1\cup J_2^\prime$ are linearly independent.
\end{proof}
The following theorem is particular case of \cite[Theorem 11.4]{Lectures_on_Mathematical_Theory} applied to the problem \eqref{eq:projection2}.
\begin{theorem}\label{theorem:representation} %(\cite[Theorem 11.4]{Lectures_on_Mathematical_Theory})
	Let ${p}\in \setD$ and $w\notin C({p})$. Then there exist numbers $\lambda_i$, $i=I_1\cup I_2$, not all zero, $\lambda_i\geq 0$, $i\in I_2$, $\lambda_i G_i(P_{C({p})}({w}),{p})=0$, $i\in I_2$ such that
	\begin{equation*}
		{w}-P_{C({p})}({w})=\sum_{I_1\cup I_2} \lambda_i g_i({p}).
	\end{equation*}
\end{theorem}

\begin{lemma}\label{lemma:linearly_independent_equalities}
	Let multifunction $\multifC$ given by \eqref{M(p,v)} satisfy RCRCQ at ${\bar{p}}\in \setD$ and $\bar{x}\in C(\bar{p})$. Assume that $C(p)\neq \emptyset$ for $p\in U_0(\bar{p})$. Then there exists a neighbourhood $U(\bar{p})$ such that for all $p\in U(\bar{p})$,  $w\notin C(\bar{p})$  we have
	\begin{equation*}
		w-P_{C(p)}(w)=\sum_{i \in I_1^0} \tilde{\lambda}_i(w,p) g_i(p)+\sum_{I_2^0(w,p)} \tilde{\lambda}_i(w,p) g_i(p),
	\end{equation*}
	where $\tilde{\lambda}_i(w,p)\in \mathbb{R}$, $i\in I_1^0$, $\tilde{\lambda}_i(w,p)>0$, $i\in I_2^0(w,p)$, $I_1^0\subset I_1$, $I_2^0(w,p)\subset I_2$, $g_i(p)$, $i\in I_1^0\cup I_2^0(w,p)$, $p\in U(\bar{p})$ are linearly independent.

\end{lemma}
\begin{proof}
	By Proposition \ref{propostion:I_1-linearly_independent}, there exists a neighbourhood $U(\bar{p})$ such that for all $p\in U(\bar{p})$ we have
	\begin{align*}
		&\{ x\ |\ \langle x \ |\ g_i(p)\rangle = f_i(p),\ i\in I_1,\ \langle x \ |\ g_i(p)\rangle \leq  f_i(p),\ i\in I_2 \}\\
		&=\{ x\ |\ \langle x \ |\ g_i(p)\rangle = f_i(p),\ i\in I_1^0,\ \langle x \ |\ g_i(p)\rangle \leq  f_i(p),\ i\in I_2 \},
	\end{align*}
	where $I_1^0\subset I_1$ and $g_i(p)$, $i\in I_1^0$ are linearly independent. Take $p\in U(\bar{p})$, $w\notin C(p)$. By Theorem \ref{theorem:representation},
	\begin{equation*}
		w-P_{C(p)}(w)=\sum_{i \in I_1^0} \lambda_i g_i(p)+\sum_{I_2} \lambda_i g_i(p),
	\end{equation*}
	where $\lambda_i\in \mathbb{R}$, $i\in I_1^0$ and $\lambda_i\geq 0$, $i\in I_2$. By Lemma \ref{lemma:positive_combination2}, there exists $I_2^0(w,p)\subset I_2$ and $\tilde{\lambda}_i(w,p)\in \mathbb{R}$, $i\in I_1^0$ and $\tilde{\lambda}_i(w,p)>0$, $i\in I_2^0(w,p)$ such that
	\begin{equation*}
		w-P_{C(p)}(w)=\sum_{i \in I_1^0} \tilde{\lambda}_i(w,p) g_i(p)+\sum_{I_2^0(w,p)} \tilde{\lambda}_i(w,p) g_i(p).
	\end{equation*}
	and $g_i(p)$, $i\in I_1^0\cup I_2^0(w,p)$ are linearly independent.
\end{proof}

\bibliographystyle{plain}
\bibliography{mybibfile}

\end{document}